\documentclass{amsart}
\usepackage{lmodern}
\usepackage[capitalize]{cleveref}
\usepackage{amssymb,amsmath,amsthm}
\usepackage{mathtools}
\usepackage{setspace}
\usepackage{tikz-cd}
\usepackage[T1]{fontenc}
\usepackage[utf8]{inputenc}
\usepackage{textcomp} 
\usepackage{stmaryrd}
\usepackage{datetime}
\usepackage{todonotes}
\usepackage{xcolor}
\usepackage{mdframed}
\providecommand{\tightlist}{%
  \setlength{\itemsep}{0pt}\setlength{\parskip}{0pt}}
\usepackage{thmtools}
\usepackage{enumitem}
\usepackage[backend=biber,url=false,maxbibnames=9]{biblatex}
\declaretheorem{theorem}
\declaretheorem[sibling=theorem]{lemma}
\declaretheorem[sibling=theorem]{proposition}
\declaretheorem[sibling=theorem]{corollary}
\declaretheorem{question}
\declaretheorem[numberwithin=theorem]{claim}
\declaretheorem[style=definition]{definition}
\newcommand{\M}{\mathcal{M}}
\newcommand{\A}{\mathcal{A}}
\newcommand{\B}{\mathcal{B}}
\newcommand{\C}{\mathcal{C}}
\newcommand{\N}{\mathcal{N}}
\newcommand{\tm}{\text{-}}

\renewcommand{\L}{\mathcal{L}}
\newcommand{\mc}[1]{\mathcal{#1}}
\newcommand{\ol}[1]{\bar{#1}}
\newcommand{\LR}{\iff}

\DeclareMathOperator{\restrict}{\upharpoonright}
\newcommand{\Pinf}[1]{\Pi^{\mathrm{in}}_{#1}}
\newcommand{\Sinf}[1]{\Sigma^{\mathrm{in}}_{#1}}
\newcommand{\Dinf}[1]{\Delta^{\mathrm{in}}_{#1}}

\newcommand{\PA}{\mathrm{PA}}
\renewcommand{\phi}{\varphi}
\newcommand{\bigwwedge}{%
  \mathop{
    \mathchoice{\bigwedge\mkern-15mu\bigwedge}
               {\bigwedge\mkern-12.5mu\bigwedge}
               {\bigwedge\mkern-12.5mu\bigwedge}
               {\bigwedge\mkern-11mu\bigwedge}
    }
}

\newcommand{\bigvvee}{%
  \mathop{
    \mathchoice{\bigvee\mkern-15mu\bigvee}
               {\bigvee\mkern-12.5mu\bigvee}
               {\bigvee\mkern-12.5mu\bigvee}
               {\bigvee\mkern-11mu\bigvee}
    }
}
\newmdtheoremenv[backgroundcolor=cyan]{theorem-prove}{Theorem}[theorem]
\newmdtheoremenv[backgroundcolor=cyan]{lemma-prove}{Lemma}[theorem]
\newmdtheoremenv[backgroundcolor=cyan]{proposition-prove}{Proposition}[theorem]
\newmdtheoremenv[backgroundcolor=cyan]{corollary-prove}{Corollary}[theorem]

\newmdtheoremenv[backgroundcolor=yellow!40]{theorem-check}{Theorem}[theorem]
\newmdtheoremenv[backgroundcolor=yellow!40]{lemma-check}{Lemma}[theorem]
\newmdtheoremenv[backgroundcolor=yellow!40]{proposition-check}{Proposition}[theorem]

\def\abar{{\bar{a}}}
\def\bbar{{\bar{b}}}

\def\a{\alpha}
\def\b{\beta}

\def\om{\omega}

\title{The structural complexity of models of arithmetic}
\author{Antonio Montalb\'an}
\address{Department of Mathematics, University of California, Berkeley}
\email{antonio@math.berkeley.edu}
\author{Dino Rossegger}
\address{Department of Mathematics, University of California, Berkeley {\normalfont and} Institute of Discrete Mathematics and Geometry, Technische Universit\"at Wien}
\email{dino@math.berkeley.edu}
\thanks{
The first author was supported by NSF grant DMS-1363310.
The work of the second author was supported by the European Union's Horizon 2020 Research and Innovation Programme under the Marie Sk\l{}odowska-Curie grant agreement No. 101026834 — ACOSE}

\begin{document}
\begin{abstract}
  We calculate the possible Scott ranks of countable models of Peano arithmetic. We show
  that no non-standard model can have Scott rank less than $\omega$ and that
  non-standard models of true arithmetic must have Scott rank greater than
  $\omega$. Other than
  that there are no restrictions. By giving a reduction via $\Dinf{1}$
  bi-interpretability from the class of linear orderings to the canonical
  structural $\omega$-jump of models of an arbitrary completion $T$ of $\PA$ we
  show that every countable ordinal $\alpha>\omega$ is realized as the Scott rank of
  a model of $T$.
\end{abstract}
  \maketitle
The structural complexity of countable structures has been an active and deep
area of research at the intersection of model theory, descriptive set theory
and computability theory for more than half a century. One of the early results
responsible for the interest in this area is Scott's theorem~\cite{scott1963}
that every countable structure has a sentence defining it up to isomorphism
among countable structures in the infinitary logic $L_{\omega_1\omega}$.
Combining this with Vaught's work~\cite{vaught1974} on the Lopez-Escobar
theorem~\cite{lopez-escobar1969} one gets a connection with descriptive set
theory. Vaught showed that the sets of models of 
$\Pinf{\alpha}$ formulas are precisely the $\pmb\Pi^0_\alpha$
isomorphism-invariant sets in the Borel hierarchy on the space of structures. Thus, not only is
every isomorphism class of a countable structure Borel, but calculating the quantifier complexity of
a structures Scott sentence gives a measure of the complexity of its
isomorphism class in the Borel
hierarchy. 

Computability theorists use another approach to 
measure the complexity of countable structures. Say a structure $\A$ is \emph{uniformly $\pmb
\Delta^0_\alpha$-categorical} if there is a Turing operator $\Phi$ and an oracle
$X\subseteq \omega$ such that for all $\B,\C$ isomorphic to $\A$, $\Phi^{(X\oplus
\B\oplus \C)^{(\alpha)}}$ is an isomorphism between $\B$ and
$\C$. It follows from results of Ash~\cite{ash1986} that a structure is uniformly $\pmb\Delta^0_\alpha$-categorical if
and only if its automorphism orbits are definable by $\Sinf{\alpha}$ formulas.
The usual proof of Scott's theorem builds a Scott sentence for a structure out of the defining
formulas of its automorphism orbits. This shows that such a structure
must have a $\Pinf{\alpha+1}$ Scott sentence. 
On the other hand, Montalb\'an~\cite{montalban2015}
showed that a structure with a $\Pinf{\alpha+1}$ Scott sentence must have all automorphism
orbits $\Sinf{\alpha}$ definable. This unifies the two approaches and gives
a robust notion of Scott rank.
  \begin{theorem}[\cite{montalban2015}]\label{thm:robusterSR}
The following are equivalent for countable $\A$ and $\alpha<\omega_1$. 
\begin{enumerate}
\tightlist
\item  Every automorphism orbit of $\A$ is $\Sinf{\alpha}$-definable without parameters.
\item  $\A$ has a $\Pinf{\alpha+1}$ Scott sentence.
\item  $\A$ is uniformly $\pmb \Delta^0_\alpha$-categorical.
\item  The set of presentations of $\A$ is $\pmb \Pi^0_{\alpha+1}$ in the Borel hierarchy.
\item\label{it:alphafree}  No tuple in $\A$ is $\alpha$-free.
\end{enumerate}
The least $\alpha$ satisfying the above is the \emph{(parameterless) Scott rank} of $\A$.
\end{theorem}
We already gave a quick introduction to all of the notions appearing in
\cref{thm:robusterSR} except \cref{it:alphafree}. This is a combinatorial condition that
is useful in arguments. It is based on the back-and-forth relations. Let us
define these formally. Our
definitions follow Montalb\'an's upcoming book~\cite{montalban2021} and we refer
to it for a thorough treatment of the back-and-forth relations and other
notions used here.
\begin{definition}
  Given an ordinal $\alpha$, $\tau$-structures $\A$ and $\B$ and tuples $\bar
  a\in A^{<\omega}$, $\bar b\in B^{<\omega}$ define the \emph{$\leq_\alpha$
  back-and-forth relations } inductively as follows.
  \[ (\A,\bar a)\leq_\alpha (\B,\bar b)\LR \forall \beta<\alpha \forall \bar
    d\in B^{<\omega} \exists \bar c\in A^{<\omega} (\B,\bar b\bar d )\leq_\beta
  (\A,\bar a\bar c)\]
  For the base case let $(\A,\bar a)\leq_0(\B,\bar b)$ if $\bar a$ and $\bar
  b\restrict |\bar a|$ satisfy the same quantifier-free $\tau_{|\bar
  a|}$-formulas.
\end{definition}
Karp~\cite{karp1965b} showed that, for $\alpha>0$, $(\A,\bar a)\leq_\alpha (\B,\bar b)$ if and only if every $\Pinf{\alpha}$-formula that is true of $\abar$ in $\A$ is true true of $\bbar$ in $\B$.

In this article we will mostly look at the case in the above definition when
$\A=\B$. In this case we will write the shorthand $\bar a\leq_\alpha \bar b$
for $(\A,\bar a)\leq_\alpha (\A,\bar b)$.
\begin{definition}
A tuple $\bar a$ in ${\mathcal{A}}$ is \emph{$\alpha$-free} if
$$ \forall (\beta<\alpha) \forall \bar b \exists \bar a'\bar b' (\bar a\bar b \leq_\beta \bar a'\bar b'\land \bar a \not\leq_\alpha \bar a').$$
\end{definition}
One of the main questions about structural complexities is how complicated
structures in natural classes can be. More formally, we want to calculate the
possible Scott ranks in natural classes of structures. Makkai~\cite{makkai1981}
defined the following.
\begin{definition}
  Let $T$ be an $L_{\omega_1\omega}$-sentence. The \emph{Scott spectrum} of $T$
  is the set 
  \[ SS(T)=\{\alpha<\omega_1: \alpha\text{ is the Scott rank of a countable
  model of T}\}.\]
\end{definition}
While this definition is about models of $L_{\omega_1\omega}$-sentences we can
also use it for measuring the complexity of structures in elementary classes,
as every first-order theory can be viewed as a sentence of $L_{\omega_1\omega}$ by
taking the conjunction of all the sentences in the theory.

The purpose of this article is to study the complexity of models of Peano
arithmetic ($\PA$) where, as usual: $\PA$ is the theory consisting of the
axioms for discrete ordered semirings and the induction scheme.
An easy argument shows that the Scott rank of the standard model $\mathbb N$ is
$1$. We show that $\mathbb N$ is the only model of $\PA$ that has Scott rank $1$
and that all other models must have Scott rank at least $\omega$. In
particular, the non-standard prime models of $\PA$ have Scott rank $\omega$ and
non-homogeneous models of $\PA$ must have Scott rank greater than $\omega$. For
non-standard models of the theory of the natural numbers, true arithmetic, we
obtain a stronger lower bound. Every such model must have Scott rank greater
than $\omega$.
Giving a reduction from the class of linear orderings via $\Dinf{1}$ bi-interpretability
to the structural $\omega$-jump of models of an arbitrary completion of $\PA$ we obtain that every Scott rank greater than $\omega$ is realized by
a model of $\PA$ and thus the following.
\begin{theorem}\label{mainthm}
  Let $T$ be a completion of $\PA$.
  \begin{enumerate}
    \tightlist 
  \item If $T=Th(\mathbb N)$, then $SS(T)=\{1\}\cup \{\alpha<\omega_1: \alpha>
      \omega\}$.
    \item If $T\neq Th(\mathbb N)$, then $SS(T)=\{\alpha<\omega_1:
      \alpha>\omega\}$.
    \end{enumerate}
    Thus, $SS(\PA)=\{1\}\cup \{\alpha<\omega_1: \alpha\geq \omega\}$.
\end{theorem}
In \cref{sec:lowerbounds} we formalize the back-and-forth relations in $\PA$ to
obtain the aforementioned lower bounds for the Scott ranks of non-standard
models. In \cref{sec:loandpa} we analyze results by Gaifman~\cite{gaifman1976}
to recover a reduction from linear orderings to models of $\PA$. It turns out
that this reduction actually provides a reduction from linear orderings to the
structural $\omega$-jumps of models of $\PA$. This and the properties of the
structural $\alpha$-jump are reviewed in \cref{sec:alphajumps}. At last we
combine all of our results to obtain a proof of \cref{mainthm}.

\section{Back-and-forth relations and Peano arithmetic}\label{sec:lowerbounds}
Throughout this article we assume that $\M$ and $\N$ are non-standard models of
$PA$, and
$M$, and $N$ are their respective universes. 
We will write $\mathbb N$ for the set of standard natural numbers in a given model and $\dot m$ for the formal term
representing the natural number $m$ in $PA$. We can code bounded subsets of a model $\M$,
\emph{$\M$-finite sets}, using elements of $\M$. Informally we let $\dot 0$ be the
code for the empty set and given an $\M$-finite set $A$, $\sum_{a\in A} 2^a$ is
its code. Using this and Cantor's pairing function $\langle
x,y\rangle=\frac{1}{2} ((x+y)^2 +3x+y)$ we can code sequences $\bar a\in
M^{<\omega}$ by the $\M$-finite set $\{ \langle i, a_i\rangle:
i<|\bar a|\}$. The length of sequences and concatenation can be defined in the
obvious way. We refer the reader to \cite{simpson2009} for formal definitions.

Let $Tr_{\Delta^0_1}$ be the truth predicate for bounded formulas, i.e., a predicate
satisfying for all bounded formulas $\phi(\overline x)$
\[PA\vdash \forall \overline x\ (Tr_{\Delta^0_1}(\ulcorner
  \phi\urcorner,\overline x) \leftrightarrow
\phi(\overline x)).\]
We define a version of the standard asymmetric back-and-forth relations where we bound the length of the tuples involved. 
The idea is that we want to be able to talk about the back-and-forth relations within models of PA, and the problem is that we don't want to have to consider tuples of non-standard length. 
Inductively define {\em bounded asymmetric $n$-back-and-forth relations with bound $a$}, denoted $\leq_n^a$, for all $n\in\omega$ as follows:
\begin{align*}
  \bar u\leq^a_0 \bar v& \Leftrightarrow \forall (x\leq |\bar u|)(
  Tr_{\Delta^0_1}(x,\bar u) \to
  Tr_{\Delta^0_1}(x,\bar v))\\
  &\mathllap{\bar u \leq_{n+1}^a \bar v} \Leftrightarrow \forall
    \bar x\exists \bar y \Big(|\bar x|\leq a \to \bar v\bar x  \leq_n^{a}
      \bar u\bar y\Big)
\end{align*}
We view $\bar u\leq_n^a \bar v$ as a formula on three variables $\bar u$, $\bar
v$, and $a$ that refer to elements of the model, and an outside parameter
$n\in\omega$ that is just part of the notation. Note that the $a$ parameter in
the definition of $\leq_0^a$ is technically superfluous. We just use it to
emphasize that $\leq_0^a$ is a relation in the language of $\PA$.
Thanks to the availability of codes for strings we regard the above formulas as
ternary predicates that are false if for some elements $\bar u,\bar v,a\in M$,
$\bar u$ and $\bar v$
are not codes for sequences. Defined like this, the
bounded back-and-forth relations behave as expected.
\begin{proposition}\label{prop:bfproperties}
  The bounded back-and-forth relations $\leq_n^x$ satisfy the following
  properties for all $n\in\om$:
  \begin{enumerate}
    \tightlist
  \item\label{it:closeddownwards} $PA\vdash\forall  \bar u, \bar v, a,
    b (( a\leq  b \land 
    \bar u\leq_n^b \bar v )\to  \bar u\leq_n^a \bar v)$
  \item\label{it:nested} $PA\vdash \forall  \bar u, \bar v, a ( 
    \bar u\leq_{n+1}^a \bar v \to  \bar u \leq_{n}^a  \bar v)$ 
\end{enumerate}
\end{proposition}
\begin{proof}
  \cref{it:closeddownwards} can be shown using induction on $n\in\om$. 
The case $n=0$ follows easily. 
For the $n+1$ case, take $a\leq b$, and $\bar u$ and $\bar v$ in the model. 
Then
\begin{eqnarray*}
\bar u\leq_{n+1}^b \bar v		&\iff  &  \forall    \bar x\exists \bar y \Big(|\bar x|\leq b \to \bar v\bar x  \leq_n^{b}      \bar u\bar y\Big)   	\\
			&\implies &		\forall    \bar x\exists \bar y \Big(|\bar x|\leq a \to \bar v\bar x  \leq_n^{b}      \bar u\bar y\Big)			\\
			&\implies &		\forall    \bar x\exists \bar y \Big(|\bar x|\leq a \to \bar v\bar x  \leq_n^{a}      \bar u\bar y\Big)	\iff  \bar u\leq_{n+1}^a \bar v,	
\end{eqnarray*}  
where the second line uses that $|\bar x|\leq a\implies |\bar x|\leq b$, and the third line uses the induction hypothesis.

\cref{it:nested} follows by an easy induction on the definition. 
\end{proof}

\begin{proposition}\label{prop:formalbfcorrect}
  Let $\overline a, \overline b\in M$. Then 
  \[ \overline a\leq_n \overline b \Leftrightarrow (\forall m\in\omega) \M\models
    \overline a\leq_n^{\dot m }\overline b.\]
  Furthermore, if there is $c\in M\setminus \mathbb N$ such that $\M\models
  \overline a\leq_n^c\overline b$, then $\overline a\leq_n\overline b$.
\end{proposition}
\begin{proof}
That $\overline a\leq_n \overline b \implies  \M\models   \overline a\leq_n^{\dot m }\overline b$ for $m\in\om$ follows by the same argument as in the previous proposition using $|\bar x|< \omega$ instead of $|\bar x|\leq b$.

That $\M\models  \overline a\leq_n^c\overline b\implies \overline a\leq_n\overline b$ for $c\in M\setminus \mathbb N$ also follows by the same argument as in the previous proposition now using $|\bar x|< c$ instead of $|\bar x|\leq b$ and $|\bar x|< \omega$ instead of $|\bar x|\leq a$.

Finally, suppose that $ (\forall m\in\omega) \M\models    \overline a\leq_n^{\dot m }\overline b$.
The set of all $m\in M$ for which $\M\models\overline a\leq_n^{m }\overline b$ is definable in $\M$ (by a $\forall_{2n}$ first-order formula), and it contains all $m\in\om$.
Since $M$ satisfies induction, this set must overspill and contain some $c\in M\setminus \mathbb N$.
So we have $\overline a\leq_n^{c}\overline b$.
It follows from the previous paragraph that we then have $\overline a\leq_n \overline b$.
\end{proof}
\begin{lemma}\label{lem:typesandbf}
   For every $\ol a, \ol b\in M^{<\omega}$, $\ol a\leq_\omega\ol b$ if and only
   if $tp(\ol a)=tp(\ol b)$.
\end{lemma}
\begin{proof}
  The left to right direction follows since $\ol a\leq_\omega\ol
  b$ implies that every $\Pinf{\omega}$ sentence true of $\ol a$ is true of
  $\ol b$. 
  To see the other direction assume that $tp(\ol a)=tp(\ol b)$.
  Then there is $\N\succ \M$ and $\sigma$ an automorphism of $\N$ with
  $\sigma(\ol a)=\ol b$, so in particular $(\N,\ol a)\leq_\omega (\N,\ol b)$.
  For every $n,m\in\omega$ we have that
  $\N\models \ol a \leq_n^{\dot m} \ol b$
  and thus also $\M\models \ol a \leq_n^{\dot m}\ol b$. 
  Thus, by \cref{prop:formalbfcorrect}, for every $n\in\omega$, $(\M,\ol a) \leq_n (\M,\ol b)$ and $(\M,\ol a)\leq_\omega(\M,\ol b)$ as required.
\end{proof}
\subsection{Homogeneous models}
Recall that a model $\M$ is \emph{homogeneous} if every partial elementary map
$M\to M$ extends to an automorphism. This implies that the automorphism orbits
of elements in homogeneous models are equal to their types. Some of the best
known examples of homogeneous models are atomic and recursively saturated
models.

\cref{lem:typesandbf} already implies that models of $PA$ that are not
homogeneous must have Scott rank larger than $\omega$. To see this recall  
that if $\M$ is not homogeneous, then there are $\ol a,\ol b\in M^{<\omega}$ such
that $tp(\ol a)=tp(\ol b)$ and $\ol a$ and $\ol b$ are not automorphic. 
If $SR(\M)\leq\omega$, then the automorphism orbit of $\ol b$ is
$\Sigma^{\mathrm{in}}_\omega$ definable and thus $\ol a\not\leq_\omega \ol b$,
a contradiction with \cref{lem:typesandbf}.
\begin{lemma}\label{lem:lowerboundnonhom}
  If $\M$ is not homogeneous, then $SR(\M)>\omega$.
\end{lemma}
Given a formula $\psi(y_1,\dots,y_n)=\exists x \phi(x,y_1,\dots,y_n)$, a model $\M$ and 
$a_1,\dots, a_n\in M$ recall that a Skolem term $s_\phi(a_1,\dots,a_n)$ is the
least element $b\in M$ satisfying $\phi(b,a_1,\dots, a_n)$. If $\M$ is a model
of $\PA$, then $b$ is uniquely determined if
$\M\models \psi(a_1,\dots a_n)$. If $\M\not\models \psi(a_1,\dots,a_n)$ we use
the convention that $b=0$.
In the special case where $\psi$ is parameter-free we refer to $b$ as Skolem
constant and denote it by $m_\phi$.
Consider the subset 
\[ 
N=\{ m_\phi:\phi \text{ an }L\text{-formula} \}.
\]
One can prove, using the Tarski-Vaught test, that $\N$ is an elementary substructure of $\M$. 
Furthermore $\N$ is unique up to isomorphism among models of $T=Th(\M)$: 
This is because for all formulas $\psi(y_1,...,y_k),
\varphi_1(x),...,\varphi_k(x)$, we have that $T$ decides whether
$\psi(m_{\varphi_1},...,m_{\varphi_k})$ holds or not.
Since $\N$ is a sub-model of all models of $T$,  it is the prime model of $T$. 
Furthermore, for any $\M\succeq \N$ if $n\in N$, then $aut_\M(n)$ is
a singleton as every element of $\N$ is definable in $\M$.

\begin{theorem}\label{thm:primesromega}
  Let $\N$ be a non-standard prime model of $PA$, then $SR(\N)=\omega$.
\end{theorem}
\begin{proof}
  Every prime model is atomic and thus has Scott rank at most $\omega$, as the defining
  formulas of the automorphism orbits are given by the isolating formulas of
  the types. 
  To see that $SR(\N)\geq \omega$ consider any non-homogeneous model
  $\M\succ \N$. Then by \cref{lem:lowerboundnonhom}, $SR(\M)> \omega$. Thus,
  for every $n$ there are $\bar a_0, \bar a_1\in
  \M^{<\omega}$ such that, $\bar a_0\leq_n \bar a_1$ but $\bar
  a_0\not\in aut_\M(\bar a_1)$. Fix $n$. 
  For every $m\in \omega$, 
  \[
  \M\models\exists \bar x_0,\bar x_1\ \bar x_0\neq \bar x_1 \land \bar
  x_0\leq_n^{\dot m} \bar x_1
  \]
  and by elementarity $\N\models \exists \bar x_0,\bar x_1\ \bar x_0\neq \bar x_1 \land \bar
  x_0\leq_n^{\dot m} \bar x_1$.
  Consider the set
  \[X_n=\{ b\in \N : \N\models \exists \bar x_0\bar x_1\ \bar x_0\neq \bar x_1\
  \land  \bar x_0 \leq_n^b \bar x_1\}\]
  which is a definable subset of $\N$ containing all of $\mathbb N$. Therefore,
  as $\N$ is non-standard, it must overspill and contain an element $b^*\in
  \N\setminus\mathbb N$. 
  Consider $\bar  a_0,\bar a_1\in N^{<\omega}$ such that  $\bar a_0\leq^{b^*}_n \bar a_1$.
  Then, by \cref{prop:formalbfcorrect},   $\bar a_0\leq_n \bar a_1$ but $\bar a_0\neq \bar a_1$.
  Since all the elements of $\N$ are definable,  all automorphism orbits are singletons.
  It follows that the automorphism orbit of $a_1$ is not $\Sinf{n}$ definable (as every $\Sinf{n}$ formula true of $a_1$ is also true of $a_0$).
  Hence  $SR(\N)> n$.
Thus, $\N$ does not have Scott rank less than $\omega$.
\end{proof}

\begin{theorem}\label{thm:srlowerbound}
  Let $\M$ be a model of $\PA$ such that $\M\not\models Th(\mathbb N)$, then $SR(\M)\geq \omega$.
\end{theorem}
\begin{proof}
 Let $\N$ be the elementary sub-model of $\M$ consisting of all Skolem
 constants in $\M$. In particular, $\N$ is the prime model of $Th(\M)$. 
 Towards a contradiction, suppose that  $SR(\M)=n<\omega$. 
 We then have that every automorphism orbit of $\M$ is $\Sinf{n}$ definable, and therefore whenever we have $\bar u \leq_n \bar a$, we have that $\bar u\in aut(\bar a)$.
 Fix $\bar a\in N^{<\omega}$.
 Since $\bar a$ is definable in $\M$, whenever we have $\bar u \leq_n \bar a$, we have  $\bar u=\bar a$.
 
 Consider the set
  \[ 
  X_{\abar}=\{ c \in M: \M\models \forall \bar
      u (\bar u \leq_n^c \bar a\to \bar a=\bar u)\}.
  \] 
This definable set contains all $c\in M\setminus \mathbb N$, and hence it must contain some $k\in\mathbb N$.
By elementarity, we get that $\N\models  \forall \bar u (\bar u \leq_n^{\dot{k}} \bar a\to \bar a=\bar u)$, and in particular whenever we have $(\N,\bar u) \leq_n (\N,\bar a)$, we have  $\bar u=\bar a$. 
 This is true for all $\bar a\in N^{<\omega}$, contradicting \cref{thm:primesromega}.
\end{proof}
For non-standard models of true arithmetic we get an even better lower bound.
\begin{theorem}\label{thm:srlowerboundta}
  Let $\M\models Th(\mathbb N)$ be non-standard. Then $SR(\M)>\omega$.
\end{theorem}
\begin{proof}
  Assume that $\M\models Th(\mathbb N)$. Then the elementary
  submodel $\N\preccurlyeq \M$ consisting of all Skolem constants in $\M$ is
  isomorphic to $\mathbb N$.
  Because $\mathbb N$ is standard and $SR(\mathbb N)=1$ it satisfies $\forall \bar a \bar
  a'( \forall x(\bar a\leq_1^x\bar a') \leftrightarrow \forall x(\bar
  a\leq_n^x\bar a'))$.
  Furthermore, no tuple in $\mathbb N$ is $1$-free. In particular, $\mathbb N$ satisfies
  the following version of $1$-freeness for the bounded back-and-forth
  relations:
  $\forall \bar a \exists \bar b\forall \bar a'\bar b' \exists x\left(\bar a \bar
  b\leq_0^x\bar
  a'\bar b' \to \forall y(\bar a \leq_1^y\bar a')\right)$. 
  Combining these two observations we get that $\mathbb N$ satisfies the
  following form of non-freeness:
 \begin{equation}\label{eq:standardmodelwitness}
    \forall \bar a \exists \bar b\forall \bar a'\bar b' \exists x\left(\bar a \bar
      b \leq_0^x \bar a'\bar b'
  \to \forall y (\bar a\leq_n^y \bar a')\right)\end{equation}
 Say $SR(\M)=\alpha$ with $1<\alpha\leq\omega$. In the case where
 $\alpha=\omega$ we have that all automorphism orbits are $\Sinf{\omega}$
 definable. Notice that this implies that for every automorphism orbit there is
 $n\in\omega$ such that the orbit is $\Sinf{n}$ definable and that the
 complexity of the defining formulas of the automorphism orbits is cofinal in
 $\omega$. Furthermore, recall that a tuple has $\Sinf{n}$ definable
 automorphism orbit if and only if it is not $n$-free~\cite[Lemma
 II.65]{montalban2021}. In any case, since $\alpha>1$ and $\alpha$ is the least
 such that no tuple is $\alpha$-free we get that there is a tuple $\bar a$
 that is $1$-free but not $n$-free for some $n<\omega$. So, $\M$ satisfies
 $\forall b \exists \bar a' \bar b' ( \bar a\bar b \leq_0\bar
 a'\bar b' \land \bar a\not \leq_1\bar a')$ and $(\exists \beta <n) \exists \bar
 b\forall \bar a' \bar b' (\bar a \bar b \leq_\beta \bar a'\bar b'\to \bar
 a \leq_n\bar a').$
 In particular, by the nestedness of the back-and-forth relations $\M$ does not satisfy $\exists \bar b \forall \bar a'\bar b'
 (\bar a\bar b \leq_0 \bar a'\bar b'\to \bar a\leq_n \bar a')$. This implies that $\M$ does not satisfy
 \cref{eq:standardmodelwitness}, contradicting that $\M\models Th(\mathbb N)$.

 It remains to show that $SR(\M)\neq 1$. Assume the contrary and let $a\in M\setminus \mathbb N$ with automorphism orbit defined by the $\Sinf{1}$ formula $\phi$. Then by elementarity there is
 $n\in\omega$
 such that $\dot n$ satisfies one of the disjuncts of $\phi$ and thus $\dot n$ is in the automorphism
 orbit of $a$. But this is a contradiction, since $\dot n$'s automorphism orbit is
 a singleton.
\end{proof}
It is easy to see that every homogeneous model has Scott rank at most
$\omega+1$, as every automorphism orbit is definable by the infinitary
conjunction over the formulas in its type. In the case where $T$ is not true
arithmetic we do not know whether there are
non-atomic homogeneous models of $T$ with Scott rank $\omega$.
\begin{question}
  Is there a non-atomic homogeneous model $\M$ with $SR(\M)=\omega$?
\end{question}
\section{The canonical structural $\alpha$-jump and
bi-interpretability}\label{sec:alphajumps}
One way to obtain a characterization of the Scott Spectrum of $\PA$ is to find
a reduction from a well-understood class of structures to models of $\PA$. One
particularly well-understood class is the class of linear orderings. It follows from results of
Ash~\cite{ash1986} that $SS(\mathrm{LO})=\{\alpha<\omega_1\}$.
\cref{thm:srlowerbound} shows that we can not find a reduction from linear
orderings that preserves Scott ranks. Instead we need a reduction that
preserves Scott ranks up to an additive factor of $\omega$. 
We will do this by giving a reduction via $\Dinf{1}$ bi-interpretability from
the class of linear orderings to the class of canonical structural $\omega$
jumps of models of a given completion of $\PA$ in \cref{sec:loandpa}. 
Before that we need to discuss infinitary bi-interpretability and canonical
structural $\omega$-jumps.

\subsection{Reductions via infinitary bi-interpretability}
Infinitary bi-interpretability  between structures, studied by Harrison-Trainor, Miller, and
Montalb\'an~\cite{harrison-trainor2018b}, is a weakening of the model-theoretic notion of bi-interpretability. 
Let us recall these notions.

\begin{definition}[\cite{harrison-trainor2018b}]\label{def:infint}
  A structure $\A=(A,P_0^\A,\dots)$ (where $P_i^\A\subseteq A^{a(i)}$) is
  \emph{infinitarily interpretable} in $\B$ if there are relations
  $Dom^\A_\B,\sim,R_0,R_1,\dots$, each $L_{\omega_1\omega}$ definable without
  parameters in the language of $\B$ such that
  \begin{enumerate}
    \tightlist
  \item $Dom_\A^\B\subseteq \B^{<\omega}$,
  \item $\sim$ is an equivalence relation on $Dom^\B_\A$,
  \item $R_i\subseteq (Dom^\B_\A)^{a(i)}$ is closed under $\sim$,
  \end{enumerate}
  and there exists a function $f_\A^\B: Dom_\A^\B\to \A$ which induces an
  isomorphism:
  \[ f_\A^\B: \A^\B=(Dom_\A^\B, R_0,R_1,...){/}{\sim}\cong \A.\]
  We say that $\A$ is $\Dinf{\alpha}$ interpretable in $\B$ if 
  the above relations are both $\Sinf{\alpha}$ and $\Pinf{\alpha}$ definable in
  $\B$.
\end{definition}
\begin{definition}[\cite{harrison-trainor2018b}]\label{def:infbiint}
  Two structures $\A$ and $\B$ are \emph{infinitarily bi-interpretable} if
  there are interpretations of each structure in the other such that the
  compositions
  \[ f_\B^\A\circ \tilde f_\A^\B: Dom_\B^{(Dom^\B_\A)} \to \B \text{ and
  } f_\A^\B\circ \tilde f_\B^\A: Dom_\A^{(Dom^\A_\B)} \to \A\] 
  are $L_{\omega_1\omega}$ definable in $\B$, respectively, $\A$. 
  If $\A$ and $\B$ are $\Dinf{\alpha}$ interpretable in each other and the
  associated compositions are $\Dinf{\alpha}$ definable, then we say that $\A$
  and $\B$ are \emph{$\Dinf{\alpha}$ bi-interpretable}.
\end{definition}
The following definition is a generalization of reducibility via effective
bi-interpretability. See~\cite{montalban2021a} for a thorough treatment of
effective bi-interpretability.
\begin{definition}
  A class of structures $\mathfrak C$ is \emph{reducible via infinitary
  bi-interpretability} to a class of structures $\mathfrak D$ if there are
  infinitary formulas defining domains, relations, and isomorphisms of an
  infinitary bi-interpretation so that every structure in $\mathfrak C$ is
  bi-interpretable with a structure in $\mathfrak D$ using this
  bi-interpretation. If all the formulas are $\Dinf{\alpha}$ then we say that
  $\mathfrak C$ is \emph{reducible via $\Dinf{\alpha}$ 
  bi-interpretability to $\mathfrak D$}.
\end{definition}
Two structures that are infinitary bi-interpretable behave in the same way.
A particularly striking testimony of this is the following result.
\begin{theorem}[\cite{harrison-trainor2018b}]\label{thm:autgroupsiso}
  Two structures are infinitary bi-interpretable if and only if their
  automorphism groups are Baire-measurably isomorphic.
\end{theorem}
It is easy to see that two $\Dinf{1}$ bi-interpretable structures have the same Scott
sentence complexity and thus also the same Scott rank. However, this
observation is not true for infinitary bi-interpretability in general.
Furthermore, \cref{thm:srlowerbound} shows that there is no hope in having
a reduction from linear orderings to Peano arithmetic that preserves the Scott
rank. Instead, our goal is that given a linear ordering $\L$ we produce
a model $\N_\L$ such that $\L$ and $\N_\L$ are bi-interpretable and
$SR(\N_\L)=\omega+SR(\L)$. Infinitary bi-interpretations of two structures $\A$ and $\B$
such that $SR(\A)$ is equal to $\alpha+SR(\B)$ have some interesting
properties. As we will see in \cref{cor:alphajumpbiint} they give $\Dinf{1}$ bi-interpretations
between $\A$ and $\B_{(\alpha)}$, the structural $\alpha$-jump of $\B$. 

\subsection{The canonical structural $\alpha$-jump}
The canonical structural $\alpha$-jump is an extension of the ideas developed
by Montalb\'an~\cite{montalban2009}, see \cite{montalban2021a} for a more up-to-date
exhibition including recent developments. The canonical structural jump of
a structure $\A$ is obtained by adding relation symbols for the $\Pinf{1}$ types of tuples in $A$ to
its vocabulary. It is $\Pinf{1}$ definable in $\A$ and $\Dinf{1}$
bi-interpretable with the jump of $\A$. When trying to  generalize to the $\alpha$-jump, it is
not immediately clear how these definability properties carry over. After all, for
$\alpha>1$, there might be continuum many formulas in the $\Pinf{\alpha}$-type
of a tuple. However, surprisingly we will see that these properties do carry over.
\begin{definition}\label{def:strucjump}
  Given a $\tau$-structure $\A$ and a countable ordinal $\alpha>0$ fix an injective enumeration  $(\bar
  a_i)_{i\in\omega}$ of representatives of the 
  $\alpha$-back-and-forth equivalence classes in $\A$. The \emph{canonical structural
  $\alpha$-jump} $\A_{(\alpha)}$ of $\A$ is the structure in the vocabulary
  $\tau_{(\alpha)}$ obtained by adding to
  $\tau$ relation symbols $R_i$ interpreted as
  \[ \bar b\in R_i^{\A_{(\alpha)}}\LR \bar a_i\leq_\alpha \bar b.\]
  We will use the convention that $\A_{(0)}=\A$.
\end{definition}
Notice that the canonical structural $\alpha$-jump of a structure is only
unique up to choice of enumeration of all the $\Pinf{\alpha}$ types. When
working with the jump we always have a fixed enumeration in mind. This does not
impose a strong restriction as for two enumerations the two resulting
candidates for the canonical structural $\alpha$-jump are $\Dinf{1}$
bi-interpretable.

\begin{proposition}\label{prop:structjumpformulas}
  Let $\A$ be a $\tau$-structure and $\phi$
  be a $\Sinf{\alpha+1}$ $\tau$-formula. Then there is
  a $\Sinf{1}$ $\tau_{(\alpha)}$-formula $\psi$ such that
\[  
\A_{(\alpha)} \models \forall \bar x \ (\phi(\bar x) \leftrightarrow \psi(\bar x)).
  \]
\end{proposition}
\begin{proof}
Since a $\Sigma_{\alpha+1}^{\mathrm{in}}$ $\tau$-formula is a disjunction of
existentially quantified $\Pi_{\alpha}^{\mathrm{in}}$ $\tau$-formulas, it is enough to prove that every $\Pi_{\alpha}^{\mathrm{in}}$ $\tau$-formula $\varphi$ is equivalent to a $\Sigma_1^{\mathrm{in}}$ $\tau_{(\alpha)}$-formula $\psi$.
  Let $(\bar a_i)_{i\in\omega}$ be the enumeration of representatives of
  $\alpha$-back-and-forth classes
  used to generate $\A_{(\alpha)}$. Given $\phi$ let $I_{\phi}=\{i: \A\models
  \phi(\bar a_i)\}$. We claim that
  $\psi=\bigvvee_{i\in I_\phi} R_i$ is as required.

  The proof of the claim follows easily from the definition of $\psi$. 
  Suppose first that $\A\models \phi(\bar a)$.
  Then $\abar\equiv_\a \abar_i$ for some $i\in I$.
  It follows that $\A_{(\alpha)}\models R_i(\bar a)$  and so  $\A_{(\alpha)}\models \psi(\bar a)$. 
  Conversely,  if  $\A_{(\alpha)}\models \psi(\bar a)$, then $\A_{(\alpha)}\models R_i(\bar a)$ for some $i\in I$.
  It follows that $\abar_i\leq_\a \abar$.
  Since $\A\models \varphi(\abar_i)$ and $\varphi$ is $\Pi_{\alpha}^{\mathrm{in}}$, we have that $\A\models \varphi(\abar)$ too.
\end{proof}

Let $\Gamma$ be a set of formulas. Then $\Gamma$ is $\Pinf{\alpha}$-supported
in $\A$ if there is a $\Pinf{\alpha}$ formula $\phi$ such that 
\[ \A\models \exists \bar x \phi(\bar x) \land \forall \bar x\left(\phi(\bar x)\to
\bigwwedge_{\gamma\in \Gamma} \gamma(\bar x)\right).\]
A proof of the following fact appeared in~\cite{montalban2021}.
\begin{proposition}[{\cite[Lemma
  II.62]{montalban2021}}]\label{prop:pialphasupported}
  For every ordinal, every structure $\A$ and every tuple $\bar a\in A^{<\omega}$,
  $\Pinf{\alpha}\tm tp^\A(\bar a)$ is $\Pinf{\alpha}$-supported in $\A$.
\end{proposition}
Note that in \cref{prop:pialphasupported} the supporting formula $\phi$ is indeed
equivalent to the $\Pinf{\alpha}\tm tp^\A(\bar a)$ as it is itself
$\Pinf{\alpha}$. In other words, $\A\models\phi(\bar b)\iff \bar
a\leq_\alpha\bar b$.
As a consequence we get a syntactic definition for the canonical structural
$\alpha$-jump, similarly to the structural $1$-jump.
\begin{corollary}\label{cor:structjumpint}
  For any structure $\A$ and non-zero ordinal $\alpha$, $\A_{(\alpha)}$ is
  $\Dinf{\alpha+1}$ interpretable in $\A$.
\end{corollary}
\begin{proof}
  All the relations $R_i^{\A_{(\alpha)}}$ are $\Pinf{\alpha}$ definable in
  $\A$. These definitions together with $Dom_{\A_{(\alpha)}}^\A=A$ and $\sim$ the graph of the
  identity function yield a $\Dinf{\alpha+1}$ interpretation of $\A_{(\alpha)}$
  in $\A$.
\end{proof}
\cref{prop:pialphasupported} also lets us proof the dual of
\cref{prop:structjumpformulas}.

\begin{proposition}\label{prop:structjumpformdual}
  Let $\A$ be a $\tau$-structure and $\phi$ be a $\Sinf{1}$
  $\tau_{(\alpha)}$-formula. Then there is a $\Sinf{\alpha+1}$ $\tau$-formula
  $\psi$ such that for all $\bar a\in A^{<\omega}$
  \[ \A\models \psi(\bar a) \LR \A_{(\alpha)}\models \phi(\bar a).\]
\end{proposition}
\begin{proof}
  To obtain $\psi$ from $\phi$ simply replace each occurence of the relation 
  $R_i$ with the supporting formula of the $\Pinf{\alpha}$ type of $\bar a_i$.
  Clearly the resulting formula is $\Sinf{\alpha+1}$ and 
  $\A_{(\alpha)}\models \phi(\bar a)$ if and only if $\A\models\psi(\bar a)$.
\end{proof}
Combining everything we have proven about the structural $\alpha$-jump so far,
the following Corollary may not be very surprising. It is however quite useful
as we will see in \cref{sec:loandpa}.
\begin{corollary}\label{cor:alphajumpbiint}
  For all countable ordinals $\alpha$ and $\beta$, the following are
  equivalent.
  \begin{enumerate}
    \item $\A_{(\gamma)}$ is $\Dinf{1}$ bi-interpretable with $\B_{(\alpha)}$.
    \item\label{it:dalphaasym} $\A$ is infinitary bi-interpretable with $\B$ such that
      \begin{enumerate}
        \item\label{it:dalphaasym_ainb} the interpretation of $\A$ in $\B$ and $f_\B^\A\circ \tilde f_\A^\B$ are
      $\Dinf{\alpha+1}$ in $\B$,
        \item\label{it:dalphaasym_bina} the interpretation of $\B$ in $\A$ and
      $f_\A^\B\circ \tilde f_\B^\A$ are $\Dinf{\gamma+1}$ in $\A$,
    \item\label{it:dalphaasym_typesainb} for every $\bar a \in Dom_{\A}^{\B}$,
      $\{\bar c: (\A^\B,\bar c)\models\Pinf{\gamma}\tm tp^{\A^\B}(\bar a)\}$ is $\Dinf{\alpha+1}$ definable in
      $\B$,
        \item\label{it:dalphaasym_typesbina} for every $\bar b \in Dom_{\B}^{\A}$,
          $\{\bar c: (\B^\A,\bar c)\models\Pinf{\alpha}\tm tp^{\B^\A}(\bar b)\}$ is $\Dinf{\gamma+1}$ definable in
      $\A$.
      \end{enumerate}
  \end{enumerate}
\end{corollary}
\begin{proof}
  Assume that $\A_{(\gamma)}$ is $\Dinf{1}$ bi-interpretable with
  $\B_{(\alpha)}$.
  From \cref{cor:structjumpint} we get that $\B_{(\alpha)}$ is $\Dinf{\alpha+1}$ interpretable
  in $\B$ and hence if $\A_{(\gamma)}$ is $\Dinf{1}$ interpretable in
  $\B_{(\alpha)}$, then it is $\Dinf{\alpha+1}$ interpretable in $\B$. In
  particular, $\A$ is $\Dinf{\alpha+1}$ interpretable in $\B$.
  Similarly, as $f_\B^\A\circ\tilde{f}^\B_\A$ is $\Dinf{1}$ definable in
  $\B_{(\alpha)}$, it is $\Dinf{\alpha+1}$ in
  $\B$. Thus, we get \cref{it:dalphaasym_ainb} and
  \cref{it:dalphaasym_typesainb}. \cref{it:dalphaasym_bina} and
  \cref{it:dalphaasym_typesbina} follow by a symmetric argument.

  Assuming all the items in \cref{it:dalphaasym} we get that $\A_{(\gamma)}$ is
  $\Dinf{\alpha}$ interpretable in $\B$ and $\B_{(\alpha)}$ is $\Dinf{\gamma}$
  interpretable in $\A$. By \cref{prop:structjumpformulas} we get that
  $\A_{(\gamma)}$ is $\Dinf{1}$ interpretable in $\B_{(\alpha)}$ and vice
  versa, that $f_\B^\A\circ\tilde{f}^\B_\A$ is $\Dinf{1}$ definable in
  $\B_{(\alpha)}$ and that $f_\A^\B\circ \tilde{f}_\B^\A$ is $\Dinf{1}$
  definable in $\A_{(\gamma)}$. Thus, $\A_{(\gamma)}$ and $\B_{(\alpha)}$ are
  $\Dinf{1}$ bi-interpretable.
\end{proof}
Using \cref{lem:typesandbf} we get that the canonical structural $\omega$-jump of
models of Peano arithmetic has a model theoretic flavor.
\begin{proposition}
  For $\N \models\PA$, the structure $(\N, (S_n)_{n\in\omega})$, where
  $(S_n)_{n\in\omega}$ is a listing of the types realized in $\N$, is
  the canonical structural $\omega$-jump of $\N$.
\end{proposition}
At last we need to relate the Scott rank of a structure with the Scott rank of
its jump. For this we need to study how the back-and-forth relations interact.
\begin{proposition}\label{prop:bfstructjump}
  Let $\A$ be a structure and $\alpha,\beta<\omega_1$ where $\beta>0$.
  Then $(\A_{(\alpha)},\bar a)\leq_{\beta} (\A_{(\alpha)},\bar b)\LR (\A,\bar
a)\leq_{\alpha+\beta}
  (\A,\bar b)$. 
\end{proposition}
\begin{proof}
  The proposition is proven by transfinite induction on $\beta$.
  The successor and limit cases are actually trivial, and the only case that matters is the base case $\b=1$. 
  
  Say $(\A,\bar a)\leq_{\alpha+1}(\A,\bar b)$, then $\Sinf{\alpha+1}\tm tp^{\A}(\bar
  b)\subseteq \Sinf{\alpha+1}\tm tp^{\A}(\bar a)$. By
  \cref{prop:structjumpformulas,prop:structjumpformdual},
  $\Sinf{1}\tm tp^{\A_{(\alpha)}}(\bar b)\subseteq \Sinf{1}\tm tp^{\A_{(\alpha)}}(\bar a)$ and thus
  $(\A_{(\alpha)},\bar a)\leq_1 (\A_{(\alpha)},\bar b)$.

  On the other hand assume that $(\A_{(\alpha)},\bar
  a)\leq_1(\A_{(\alpha)},\bar b)$. We have that $(\A,\bar a)\leq_{\alpha+1}
  (\A,\bar b)$ if and only if for all $\bar d$ there is $\bar c$ such that
  $(\A,\bar b\bar d)\leq_\alpha (\A,\bar a\bar c)$. 
  Fix a tuple $\bar d$ and let $\bar a_i$ be such that $\bar
  b\bar d\equiv_\alpha \bar a_i$. Then $(\A_{(\alpha)},\bar b\bar d)\models
  R_i(\bar b\bar d)$, and in particular $(\A_{(\alpha)},\bar
  b)\models \exists \bar x R_i(\bar b\bar x)$. By assumption that $(\A_{(\alpha)},\bar
  a)\leq_1(\A_{(\alpha)},\bar b)$, also $(\A_{(\alpha)},\bar
  a)\models \exists \bar x R_i(\bar a\bar x)$. Pick a witness $\bar c_0$ for
  $\bar x$. Then $\Pinf{\alpha}\tm tp^\A(\bar a\bar c_0)\supseteq
  \Pinf{\alpha}\tm tp^\A(\bar b\bar d)$ and thus $(\A,\bar b\bar
  d)\leq_\alpha(\A,\bar a \bar c_0)$ as required.
\end{proof}
Assume that a tuple $\bar a$ from $\A$ is $(\alpha+\beta)$-free where $\beta\geq 1$. Then in
particular for all $\gamma<\beta$
\[ \forall \bar b \exists \bar a'\bar b' \left(\bar a\bar b\leq_{\alpha+\gamma}\bar
  a'\bar b'\land
\bar a\not\leq_{\alpha+\beta}\bar a'\right)\]
and hence by \cref{prop:bfstructjump}, \[(\forall \gamma<\beta )\forall \bar b\bar b\exists\bar a'\bar b'
\left((\A_{(\alpha)},\bar a\bar b)\leq_{\gamma} (\A_{(\alpha)},\bar a'\bar b') \land
(\A_{(\alpha)},\bar a)\not\leq_\beta (\A_{(\alpha)},\bar a')\right).\]
So, $\bar a$ is $\beta$-free in $\A_{(\alpha)}$. That $\bar a$ being $\beta$-free
in $\A_{(\alpha)}$ implies that $\bar b$ is $(\alpha+\beta)$-free in $\A$ follows
from the nestedness of the back-and-forth relations. That is, if for some $\bar
a, \bar b$, and $\beta$, $\bar a\leq_\beta \bar b$, then for all $\gamma<\beta$,
$\bar a\leq_\gamma \bar b$. Hence, we get the following from
\cref{it:alphafree} in \cref{thm:robusterSR}.
\begin{corollary}\label{cor:srjump}
  For any structure $\A$ and non-zero $\alpha,\beta<\omega_1$,
  $SR(\A)=\alpha+\beta$ if and only if $SR(\A_{(\alpha)})=\beta$.
\end{corollary}
\section{Reducing linear orderings to models of $\PA$}\label{sec:loandpa}
The goal of this section is to complete the proof of \cref{mainthm} by showing
that for every completion $T$ of $\PA$ and every countable ordinal $\alpha$ bigger than
$\omega$, $\alpha$ is the Scott rank of a model of $T$. The main missing part
is the following theorem which is the main result of this section.
\begin{theorem}\label{thm:loandpa}
 For any completion $T$ of $\PA$, the class of linear orderings is
 reducible via $\Dinf{1}$ bi-interpretability to the canonical structural
 $\omega$-jumps of models of $T$.
\end{theorem}
To prove \cref{thm:loandpa} we use a classical construction of
Gaifman~\cite{gaifman1976}. The following result is a special case of one
of his results, see~\cite[Section 3.3]{kossak2006} for more details.
\begin{theorem}[cf. \cite{gaifman1976}]\label{thm:gaifman}
  For every completion $T$ of $\PA$ and every linear order $\L$, there is
  a model
  $\N_\L\models T$ such that the automorphism groups of $\L$ and $\N_\L$
  are isomorphic. 
\end{theorem}
In light of \cref{thm:autgroupsiso} this result suggests that there is an
infinitary bi-interpretation between $\L$ and $\N_\L$. Coskey and Kossak~\cite{coskey2010} used Gaifman's
construction to obtain a Borel reduction from linear orderings to models of $T$
and thus obtained that the isomorphism relation on the models of $T$ is Borel
complete. Analyzing Gaifman's construction we will see that this $\L$ is $\Dinf{1}$ bi-interpretable with the
canonical structural $\omega$-jump of $\N_\L$. 

\subsection{Gaifman's $\L$-canonical extension}
This section follows Gaifman's construction. We will outline his proof and
refer the reader to~\cite[Section 3.3]{kossak2006} for details.
We will work with a fixed completion $T$ of $\PA$. The main ingredient of Gaifman's construction are \emph{minimal types}.
A type $p(x)$ is
\emph{minimal} if and only if $p(x)$ is
\begin{enumerate}
\tightlist    
\item \emph{unbounded}: $(t<x)\in p(x)$ for every Skolem constant, and
\item \emph{indiscernible}: for every model $\M$, and all increasing sequences
  of elements $\bar a, \bar b$ in $M$ of the same length with each element
  realizing $p(x)$, $tp^\M(\bar a)=tp^\M(\bar b)$.
\end{enumerate}
This is not Gaifman's original definition, rather a convenient
characterization. The original definition was that a type $p(x)$ is
\emph{minimal} if it is unbounded and whenever $\M\models T$ and $\M(a)$ is
a $p(x)$-extension of $\M$, then there is no $\N$ such that $\M\prec\N\prec
\M(a)$. See~\cite[Section 3.2]{kossak2006} for this and other characterizations
of minimal types.

One other property of minimal types is that they are \emph{definable} in the
sense of stable model theory. That is, for every
formula $\phi(u,v)$, there is a formula $\sigma_\phi(u)$ such that for all
Skolem constants, $\phi(t,v)\in p(x) \Leftrightarrow T\vdash
\sigma_\phi(t)$.
Gaifman used these types to build $\L$-canonical extensions of models $\M$
given a linear order $\L$. We will build the $\L$-canonical extension of the
prime model $\N$, denoted by $\N_\L$. We will fix a minimal type $p(x)$ not
realized in $\N$. Gaifman~\cite[Theorem 3.1.4]{kossak2006} used Ramsey's
theorem to show that minimal types exist.
Let us point out though that one can find a minimal type recursive in
$T$~\cite[Remark below Theorem 3.1.4]{kossak2006}.

Let $p(x), q(y)$ be two definable types. Then we can
define the product $p(x)\times q(y)$ to be the type $r(x,y)$ of $(a,b)$ in $\M(a)(b)$
where $\M(a)$ is a $p(x)$-extension of $\M$ and $\M(a)(b)$ is a
$q(y)$-extension of $\M(a)$. Definability of $q(x)$ guarantees uniqueness of
$r(x,y)$, as $\phi(x,y)\in r(x,y)$ if and only if $\exists x \phi(x,y)\in q(y)$
and $\sigma_\phi(x)\in p(x)$. Furthermore, if $\M(a_1,\dots,a_n)$ is
a $p_1(x_1)\times\dots\times p_n(x_n)$-extension and $\M(b_1,\dots b_k)$ is a $p_{i_1}(x_{i_1})\times\dots \times p_{i_k}(x_{i_k})$-extension with
all $i_k$ disjoint and less than $n$, then there is a unique elementary
embedding that is the identity on $\M$ and takes $b_j$ to $a_{i_j}$. This
allows us to iterate $p(x)$-extensions. Given a linear ordering $\L$, associate
with every $l\in L$ a variable $x_l$ and a minimal type $p(x_l)$ (where
$p(x_{l_1})$ and $p(x_{l_2})$ only differ in the free variable). Fix the prime
model $\N$ of $T$ and let 
$\N_\L$ be the structure obtained by taking the direct limit of all
$p(x_{l_1})\times\dots \times p(x_{l_k})$-extensions of $\N$ for every ascending
sequence $l_1<\dots<l_k$ in $\L$. The so obtained structure $\N_\L$ is commonly
referred to as an $\L$-canonical extension of $\N$. We will refer to the
elements of $\L$ and their corresponding elements in $\N_\L$ as the generators
of $\N_\L$.

\subsection{Interpreting $\L$ in $\N_\L$}
  The prime model $\N$ of any completion $T$ of $\PA$ has a copy whose elementary
  diagram -- the set $\bigoplus_{i,j\in\omega} \{\langle a_1,\dots,a_j
  \rangle: \N\models \phi^j_i(\bar a_1,\dots,\bar a_j) \}$ where $\phi_i^j$ is
  the $i$th formula in the language of arithmetic with $j$ free variables -- is
  $T$-computable. We say that such $\N$ is $T$-decidable. 

  Something similar can be observed for the models $\N_\L$. Fix an enumeration
  $(s_i)_{i\in\omega}$ of the Skolem terms of $T$. Let
$Var=\{x_i:i\in L\}$, and, given $\L$, let $\lambda:Var^{<\omega} \to
Var^{<\omega}$ be the function taking tuples of variables and outputting them
such that their indices form an $\L$-ascending sequence. The canonical copy
  of $\N_\L$ is given by the quotient of the Skolem terms with parameters the
  generators under the equivalence $\sim$ given by
  \[s_m(\bar x)\sim s_n(\bar x') \LR
    \left(\prod_{i<|\bar x|+|\bar x'|}p(\lambda(\bar x^\smallfrown\bar x')_{i})
  \right)\models s_m(\bar x)=s_n(\bar x'). \]
For any first order formula $\phi$ and elements $s_1(\bar x_1),\dots, s_m(\bar
x_m)$ in $\N_\L$ we have that 
\begin{multline*} \N_\L \models \phi(s_1(\bar x_1),\dots, s_m(\bar x_m)) \\\LR \left(\prod_{i<|\bar
  x_1|+\dots+ |\bar x_m|} p(\lambda({\bar x_1}^\smallfrown\dots^\smallfrown \bar
x_m)_i)\right)\models  \phi(s_1(\bar x_1),\dots , s_m(\bar x_m)).\end{multline*}

  We will interpret this canonical presentation of $\N_\L$ in $\L$ using
  relations in $\omega\times \N_\L^{<\omega}$.\footnote{We use $\omega\times
  \N_\L^{<\omega}$ instead of $\N_\L^{<\omega}$ for conceptual reasons. Given
  $R\subseteq \omega\times \N_\L^{<\omega}$ we can effectively pass to
  a relation $R'\subseteq \N_\L^{<\omega}$ and vice versa by letting
  $R'=\{\underbrace{a \dots a}_{n \text{ times}} b\bar c: a,b\in \N_\L, (n,\bar
c)\in R\}$.}
We let $Dom_{\N_\L}^\L=L\cup \{\langle n,l_1,\dots l_m\rangle: s_n(x_{l_1}\dots
x_{l_m})\in \N_\L\}$ and the relation symbols and $\sim$ to be interpreted in
the obvious way from $\N_\L$. It is not immediate that the so defined
relations are $\Dinf{1}$-definable in $\L$. We will use the relativization of
a result obtained by

 Ash, Knight, Manasse, Slaman~\cite{ash1989}, and, independently,
Chisholm~\cite{chisholm1990}. A proof for the version that we are using can be
found in~\cite{montalban2021a}.
\begin{lemma}\label{lem:urice}[\cite{montalban2021a}, cf.
  \cite{ash1989,chisholm1990}]
  Let $\A$ be a structure, $R\leq \omega \times A^{<\omega}$ a relation on $\A$
  and $X\subseteq\omega$. The
  following are equivalent:
  \begin{enumerate}
    \tightlist
  \item $R$ is uniformly relatively intrinsically $X$-c.e., that is in every copy
    $\hat \A\cong \A$, the relation $R^{\hat \A}$ is $X$-c.e.\ in $\hat \A$,
    uniformly in the copies.
  \item $R$ is definable by an $X$-computable $\Sinf{1}$ formula in the language of
    $\A$.
\end{enumerate}
\end{lemma}

Using this lemma one easily sees that the elementary diagram of $\N_\L$ is $\Dinf{1}$ interpretable in $\L$.

In order to prove \cref{thm:loandpa} we want to use
\cref{it:dalphaasym} of \cref{cor:alphajumpbiint}. Towards that we show that the
$\Pinf{\omega}$ types of tuples in $\N_\L$ are $\Dinf{1}$ definable in $\L$. By
\cref{lem:typesandbf} it is sufficient to show that all first order
types realized in $\N_\L$ are $\Dinf{1}$ definable in $\L$. Since the full first order structure of $\N_\L$
is $\Dinf{1}$ interpretable in $\L$, the sets $\{ \bar b: \bar b\models
tp^{{\N_\L}^\L}( \bar a)\}$ for $\bar a\in \N_\L$ are clearly $\Pinf{1}$ definable
in $\L$. We will show that they are also $\Sinf{1}$
definable.
\begin{lemma}
  For every $\bar a\in (Dom_{\N_\L}^\L)^{<\omega}$, $\{ \bar b: \bar b\models
tp^{{\N_\L}^\L}( \bar a)\}$ is
  $\Sinf{1}$-definable in $\L$.
\end{lemma}
\begin{proof}
  First note that if $\bar a$ is an $n$-tuple of generators of $\N_\L$, then
  $tp^{{\N_\L}^{\L}}(\bar a)=\prod_{i<n} p(x_i)$ and the elements satisfying this type
  is the $\sim$-closure of the generators of length $n$ from $\L$. It is therefore trivially
  $\Sinf{1}$ definable in $\L$. Every other element is a Skolem term with
  generators as parameters. We will view these elements as such. Recall that
  in a model of $\PA$ we can represent every finite sequence of elements by a single element, its
  code. Given $\bar a$, recall that its code is the element $c(\bar a)=\sum_{i<|a|}
  2^{\langle i,a_i\rangle}$. It is then not hard to see that $\bar b\models
  tp(\bar a)$ if and only if $c(\bar b)\models tp(c(\bar a))$. Thus, it is sufficient
  to establish the lemma for $1$-types.
  \begin{claim}\label{claim:tpsclosed}
    Let $s$ be a Skolem term and $a=s(l_1,\dots, l_n)$ where  $l_1<\dots<l_n\in L$. If
    $b=s(k_1,\dots,k_n)$ for some $k_1<\dots<k_n\in L$, then $b\models tp(a)$.
  \end{claim}
  Assuming the claim, we have that every type in ${\N_\L}^\L$ is a countable
  union of Skolem terms with parameters all ordered $L$-tuples. Let
  $(s_i)_{i\in\omega}$ be an enumeration of these Skolem terms. Then it is
  definable by a $\Sinf{1}$ formula of the form 
  \[ y\in tp(a) \Leftrightarrow \bigvvee_{i\in\omega} \exists
    x_1,\dots,x_{m(i)}\ y=s_i(x_1,\dots,
  x_{m(i)})\]
  where the $s_i$ can be viewed as a formula in the language of $\PA$ and thus
  are $\Dinf{1}$ interpretable in $\L$.
  \begin{proof}[Proof of \cref{claim:tpsclosed}]
    Note that $(k_1,\dots,k_n)\models tp(l_1,\dots,l_n)$. Say
    $b=s(k_1,\dots,k_n)$ and that $b\not\in tp(a)$, then there is $\psi$ such
    that ${\N_\L}^\L\models \psi(b)$ and ${\N_\L}^\L\not\models \psi(a)$. So 
    ${\N_\L}^\L\models \exists x ( x=s(l_1,\dots,l_n) \land \neg \psi(x) )$ and
  ${\N_\L}^\L\models \exists x ( x=s(k_1,\dots,k_n) \land \psi(x) )$,
      a contradiction as $(k_1,\dots,k_n)\models tp(l_1,\dots,l_n)$.
  \end{proof}
\end{proof}
We have thus shown the following.
\begin{lemma}
  For every linear ordering $\L$, $(\N_\L)_{(\alpha)}$ is $\Dinf{1}$
  interpretable in $\L$.
\end{lemma}
\subsection{Interpreting $\L$ in $\N_\L$}
Our goal is to show that $\L$ is $\Dinf{\omega+1}$ interpretable in $\N_\L$.
This follows from the following lemma which can be extracted from various
results of Gaifman.
\begin{lemma}\label{lem:ordertypepx}
  For any $\L$, $\{a: tp^{\N_\L}(a)=p(x)\}=L$. In particular, $(\{a:
  tp^{\N_\L}(a))=p(x)\}, \leq^{\N_\L})\cong \L$.
\end{lemma}
\begin{proof}[Proof sketch.]
To obtain a proof of this lemma we need the definition of the \emph{gap} of an element.
For fixed $\M\models PA$, let $\mathcal F$ be the set of first-order definable functions $f:
M\to M$ for which $x\leq f(x)\leq f(y)$ whenever $x\leq y$. Then for any $a\in M$
let $gap(a)$ be the smallest set $S$ with $a\in S$ and if $b\in S$,
$f\in\mc F$, and $b\leq x\leq f(b)$ or $x\leq b\leq f(x)$, then $x\in S$.
Notice that the gaps of $\M$
partition $\M$ into $\leq^\M$-intervals. We can thus obtain
an equivalence relation $a\sim b \LR gap(a)=gap(b)$
and study the order type $(\M{/}{\sim}, \leq^\M)$, the order type of the gaps of $\M$.
Notice, that if $\M$ is the prime model of $T$, then the order type of its gaps
is $1$ and unsurprisingly, the order type of the gaps of $\N_\L$ is
$1+o(\L)$~\cite[Corollary 3.3.6]{kossak2006}.

Minimal types and gaps interact in interesting ways. Minimal types are
rare, that is no two elements in the same gap can realize the same minimal
type~\cite[Lemma 3.1.15]{kossak2006}. A different characterization of a rare
type $p(x)$ is that if $a$ is an element realizing it and $b\in gap(a)$, then
$a$ is in the Skolem closure of $b$~\cite[Theorem 3.1.16]{kossak2006}. So, in particular, if the type $p(x)$ is non-principal, then in
the first gap, the gap of $0$, $p(x)$ can not be realized.
Thus, we get that $\L$ and the ordering of the elements of type $p(x)$ in
$\N^\L$ are isomorphic.

It remains to show that the elements of type $p(x)$ are precisely the
generators. Clearly the type of every generator is $p(x)$ by construction. To see that $L\supseteq \{a:tp^{\N_\L}(a)=p(x)\}$,
note that every $b\in \N_\L$ is of the form $s(l_1,\dots,l_m)$ where $s$ is
a Skolem term and $l_1,\dots,l_m$ is a sequence of generators. Then $b\in
\N(l_1,\dots,l_m)$ and thus by rarity, if $b$ had type $p(x)$, then $gap(b)\neq
gap(l_i)$ for $i<m$ and $gap(b)\neq gap(0)$. But the order type of the gaps in
$\N(l_1,\dots,l_m)$ is $1+m$ and every $l_i$ and $0$ sit in their own gap, so this is impossible.
\end{proof}
\cref{lem:ordertypepx} yields a natural interpretation of $\L$ in $\N_\L$ given by $Dom_\L^{\N_\L}=\{ a:
a\models p(x)\}$, $\sim=id$ and $\leq^{\L^{\N_\L}}=\leq^{\N_\L}$. The only
complicated relation is $Dom_{\L}^{\N_\L}$ which has a $\Pinf{\omega}$
definition given by the conjunction of all formulas in $p(x)$. We thus obtain
the required interpretation.
\begin{lemma}
  Every linear ordering $\L$ is $\Dinf{\omega+1}$ interpretable in $\N_\L$.
\end{lemma}
To finish the proof of \cref{thm:loandpa} it remains to show that $\N_\L$ and
$\L$ are infinitary bi-interpretable, that the function $f_{\L}^{\N_\L}\circ
\tilde f_{\N_\L}^\L: {\L^{(\N_\L)}}^\L\to \L$ is $\Dinf{1}$ definable in $\L$ and that
$f^\L_{\N_\L}\circ \tilde f^{\N_\L}_\L: {(\N_\L)^\L}^{(\N_\L)}\to \N_\L$ is $\Dinf{\omega+1}$ definable in
$\N_\L$. We have that $Dom_\L^{Dom_{\N_\L}^\L}=\{ \bar a\in L^{<\omega}:
tp^{(\N_\L)^\L}(\bar a)=p(x)\}=L$. Thus, by \cref{lem:ordertypepx}, $f_{\L}^{\N_\L}\circ
\tilde f_{\N_\L}^\L=id_\L$ and so trivially $\Dinf{1}$ definable in $\L$.

On the other hand $Dom_{\N_\L}^{Dom_\L^{\N_\L}}=\{\langle n, a_1,\dots,
a_m\rangle: n\in\omega, (\forall i<m) tp^{\N_\L}(a_i)=p(x)\}$ and 
\begin{multline*}Graph_{f^\L_{\N_\L}\circ \tilde f^{\N_\L}_\L}=\\\{ (\langle n, a_1\dots
    a_m\rangle ,b): \langle n, a_1\dots
    a_m\rangle\in Dom_{\N_\L}^{Dom_\L^{\N_\L}} \mathbin\& b=s_n(a_1,\dots,a_m)\}\end{multline*}
which is easily seen to be $\Dinf{\omega+1}$ definable in $\N_\L$.
We have thus shown that the interpretations $\L^{\N_\L}$ and $(\N_\L)^{\L}$
satisfy all the conditions of \cref{it:dalphaasym} in
\cref{cor:alphajumpbiint}. Hence, we obtain \cref{thm:loandpa} and can close by
finishing the proof of \cref{mainthm}.

\begin{proof}[Proof of \cref{mainthm}]
  First, recall that by \cref{thm:srlowerbound} no model of $\PA$ except the
  standard model has Scott rank $n<\omega$. By \cref{thm:primesromega}
  non-standard prime models have Scott rank $\omega$ and thus $\omega$ is the
  least element of the Scott spectrum of any completion of $\PA$ that is not
  true arithmetic. Similarly, for true arithmetic, by \cref{thm:srlowerboundta}
  the only element of the Scott spectrum below $\omega+1$ is $1$.

  To obtain models of Scott rank $\alpha$ for $\alpha>\omega$ we use \cref{thm:loandpa} that says
  that for every completion $T$ of $\PA$ and every linear order $\L$ there is a model $\N_\L$ of $T$
  such that $\L$ is $\Dinf{1}$ bi-interpretable with ${\N_\L}_{(\omega)}$. Hence,
  $SR({\N_\L}_{(\omega)})=SR(\L)$ and thus, by \cref{cor:srjump},
  $SR(\N_\L)=\omega+SR(\L)$. It remains to show that there are linear orderings
  of every Scott rank. This follows from Ash's characterization of the
  back-and-forth relations~\cite{ash1986}: $SR(\omega^\alpha)=2\alpha$,
  $SR(\omega^\alpha\cdot 2)=2\alpha+1$. For detailed calculations
  see~\cite[Proposition 19]{alvir2020}. To obtain a linear ordering of Scott
  rank $1$ one can either consider finite linear orderings or $\eta$, the order
  type of the rationals. Both are uniformly $\Dinf{1}$ categorical.
\end{proof}

\printbibliography

\end{document}